\documentclass[11pt]{article}
\usepackage[a4paper,margin=1in]{geometry}
\usepackage{amsmath,amssymb,amsthm,mathtools,enumitem}
\usepackage[utf8]{inputenc}
\usepackage{newtxtext,newtxmath}

\usepackage{hyperref}

\newtheorem{theorem}{Theorem}[section]
\newtheorem{proposition}[theorem]{Proposition}
\newtheorem{lemma}[theorem]{Lemma}
\newtheorem{corollary}[theorem]{Corollary}
\theoremstyle{definition}
\newtheorem{definition}[theorem]{Definition}

\newtheorem{problem}[theorem]{Problem}

\newcommand{\ZZ}{\mathbb{Z}}
\newcommand{\CC}{\mathbb{C}}

\newcommand{\Proj}{\mathrm{Proj}}

\DeclareMathOperator{\Span}{span}

\DeclareMathOperator{\Pic}{Pic}

\begin{document}

\title{Picard groups of completed period images and the Deng–Robles problem}
\author{Badre Mounda, Dongzhe Zheng}
\date{\today}
\maketitle

\begin{abstract}
A basic problem in the geometry of degenerating period maps is to determine whether their completed images admit an intrinsic algebraic description. For polarized variations of Hodge structure over smooth quasi-projective surfaces, Deng and Robles formulated such a problem in terms of the Kato--Nakayama--Usui completion of the period image and a conjectural Proj description involving the augmented Hodge line bundle and the boundary divisor on a smooth compactification of the base. We show that the essential obstruction to this description is divisor-theoretic: it may be expressed as a Picard-generation statement on the completed mixed period image. We prove this statement when the pure period image is one-dimensional, and consequently obtain the Deng--Robles Proj description in this case.
\end{abstract}

\section{Introduction}

Period maps encode the variation of Hodge structure in algebraic families, and their behavior near the boundary is governed by limiting mixed Hodge structures and asymptotic data. A basic problem in Hodge theory is whether the image of a degenerating period map admits a natural algebraic completion. Beyond the Hermitian setting, this leads naturally to constructions involving nilpotent orbits, logarithmic structures, and mixed period spaces.

In dimension two, Deng and Robles formulated a precise version of this problem. Let $S$ be a smooth quasi-projective surface, let $\bar S$ be a smooth projective compactification with simple normal crossings boundary, and let $\mathbb V$ be a polarized variation of Hodge structure on $S$. Their work yields a completed image
\[
Y \subset \Gamma \backslash D_{\Sigma},
\]
and asks whether $Y$ admits a Proj description in terms of the augmented Hodge line bundle on $\bar S$ and the boundary divisor.

We show that this problem is governed by the divisor theory of the completed mixed period image. More precisely, we recast the Deng--Robles problem as a Picard-generation statement for $Y$. The essential point is that the difficulty lies not only in the existence of the completed image, but in the structure of the divisor classes it carries.

Our main result treats the case in which the pure period image is one-dimensional. In this setting, the geometry of the completed image becomes sufficiently rigid to allow a complete analysis of its divisor theory. As a consequence, we establish the required Picard-generation statement and thereby verify the Deng--Robles Proj description in this case.

The proof proceeds by relating the divisor theory of $Y$ to the geometry of the induced map from the mixed period image to the pure period image. When the latter is a curve, the horizontal contribution is controlled by the geometry of the base, while the vertical contribution is governed by the geometry of the fibers together with positivity results for mixed period maps. This is where results of Green--Griffiths--Robles and Bakker--Brunebarbe--Tsimerman enter in an essential way.

The paper is organized as follows. In Section~2 we recall the precise formulation of the Deng--Robles problem in the two-dimensional setting. In Section~3 we collect the Hodge-theoretic and algebro-geometric ingredients used in the argument. In Section~4 we reformulate the Deng--Robles problem as a Picard-generation problem on the completed image. In Section~5 we analyze the case where the pure period image is a curve and establish the required Picard-generation statement. In Section~6 we combine these results to prove the main theorem.

\section{Problem Statement}

Let \(S\) be a smooth quasi-projective surface, \(\bar S\) its smooth projective compactification, with boundary
\[
Z := \bar S \setminus S = \bigcup_i S_i
\]
a simple normal crossing (SNC) divisor. Let \(\mathbb V\) be a polarized integral variation of Hodge structure (VHS) on \(S\), with associated period map
\[
\Phi : S \longrightarrow \Gamma \backslash D,
\]
where \(D\) is the period domain of a given Hodge type, \(\Gamma \subset G(\mathbb Q)\) is a discrete subgroup of the algebraic group \(G\), such that \(\Gamma\backslash D\) is a Hausdorff complex analytic space. When necessary, we may assume the monodromy is unipotent after passing to a finite étale cover.

Deng--Robles in \cite{DengRobles23} constructed a weak fan \(\Sigma^{(2)}_{\Phi,S}\) compatible with \(\Phi\) in the case \(\dim S=2\), thereby obtaining a completion in the sense of Kato--Nakayama--Usui
\[
\Phi_\Sigma : S \longrightarrow \Gamma\backslash D_\Sigma,
\]
where \(D_\Sigma\) is the mixed period domain determined by \(\Sigma\). Denote the closure of the image by
\[
Y := \overline{\Phi_\Sigma(S)} \subset \Gamma\backslash D_\Sigma.
\]

Let \(\Lambda\) be the Deligne--Griffiths--Schmid extension of the Hodge line bundle (or augmented Hodge line bundle) on \(\bar S\). In fundamental cases such as weight \((1,2)\), \(\Lambda\) coincides with the usual Hodge bundle.

\begin{problem}[Deng--Robles Question~1.8, dimension two case {\cite{DengRobles23}}]\label{prob:DR-1.8}
Under the above setting, do there exist an integer \(m>0\) and integers \(a_i\ge 0\) such that
\[
Y \;\cong\; \Proj\Big(\bigoplus_{k\ge0}
H^0\big(\bar S,\ \mathcal O_{\bar S}\big(k(m\Lambda - \sum_i a_i S_i)\big)\Big)?
\]
\end{problem}

Here \(\Proj(-)\) is polarized by \(m\Lambda-\sum a_i S_i\). In other words, the question asks: Can \(Y\) be constructed as a Proj-type algebraic space using the augmented Hodge line bundle on the base \(\bar S\) and its boundary correction?

In previous work, the following conditional result has been established: Let
\[
\bar S \xrightarrow{\,f\,} Y
\]
be the Stein factorization of the analytic closure of \(\Phi_\Sigma\) on \(\bar S\), and let \(A\) be any ample line bundle on \(Y\). If in the Néron--Severi group \(N^1(\bar S)_\mathbb Q\) we have
\begin{equation}\label{eq:Span-base}
c_1(f^*A) \in \Span_{\mathbb Q}\big(c_1(\Lambda), [S_i]\big),
\tag{\(*\)\,Span}
\end{equation}
then Problem~\ref{prob:DR-1.8} has an affirmative answer. That is, in the case \(\dim S=2\), provided \eqref{eq:Span-base} is satisfied, the description of \(Y\) can be obtained via Proj argument.

The goal of this paper is twofold:

\begin{itemize}[leftmargin=2em]
  \item First, to precisely translate condition \eqref{eq:Span-base} into a Picard generation problem for the image space \(Y\), namely a structural statement about
  \[
  \Pic_{\mathbb Q}(Y) := \Pic(Y)\otimes_\ZZ \mathbb Q;
  \]
  \item Second, in the natural case ``pure period image dimension = 1'', using tools from Green--Griffiths--Robles \cite{GGR25} and Bakker--Brunebarbe--Tsimerman \cite{BBT23}, to completely prove the Picard generation of the image space \(Y\), thereby verifying the conclusion of Problem~\ref{prob:DR-1.8} in this case. In particular, this provides a non-Hermitian example with \(\dim Z=1\).
\end{itemize}

For convenience in the subsequent exposition, we first systematically organize the necessary external tools and basic results in the next section.

\section{Preliminaries}\label{sec:prelim}

This section collects external results and basic algebro-geometric facts used in the proofs of this paper.

\subsection{Kato--Usui and Kato--Nakayama--Usui Completion}\label{subsec:KNU}

Kato--Usui in \cite{KU09} gave the classification space for polarized pure VHS in the degenerate case; Kato--Nakayama--Usui in \cite{KNU10,KNU13} generalized this to the mixed case. A brief statement is as follows.

\begin{theorem}[Kato--Usui, pure case {\cite{KU09}}]\label{thm:KU-pure}
Let \(D\) be the period domain parametrizing polarized pure Hodge structures of a given type, and \(\Gamma\subset G(\mathbb Q)\) a discrete subgroup. Given a rational fan \(\Sigma\) satisfying strong compatibility with \(\Gamma\), there exists a complex space with log structure
\[
D_\Sigma
\]
and a natural embedding \(D\hookrightarrow D_\Sigma\), and for each polarized VHS with quasi-unipotent monodromy at the boundary, its period map \(\Phi:S\to \Gamma\backslash D\) extends to a log morphism
\[
\Phi_\Sigma : \bar S^{\log}\longrightarrow \Gamma\backslash D_\Sigma
\]
agreeing with the original period map on the smooth open subset \(S\).
\end{theorem}

In the mixed case, Kato--Nakayama--Usui constructed the classification space for log mixed Hodge structures:

\begin{theorem}[Kato--Nakayama--Usui, mixed case {\cite{KNU10,KNU13}}]\label{thm:KNU-mixed}
For any given graded-polarized mixed Hodge structure type and a weak fan \(\Sigma\), there exists a mixed period domain \(D\) and its nilpotent orbits space \(D_\Sigma\), with a natural \(\Gamma\)-action, such that: for any admissible graded-polarized VMHS with monodromy compatible with \(\Sigma\), the period map \(S\to \Gamma\backslash D\) extends to
\[
\Phi_\Sigma : \bar S^{\log}\to \Gamma\backslash D_\Sigma.
\]
\end{theorem}

Usui in \cite{Usu06} proved the algebraicity and compactness of the extended image.

\begin{theorem}[Usui {\cite{Usu06}}]\label{thm:Usui}
Let \(S\) be a smooth quasi-projective variety, \(\bar S\) an SNC compactification, and \(\mathbb V\) an admissible polarized integral VMHS on \(S\) with quasi-unipotent monodromy. If there exists a weak fan \(\Sigma\) compatible with the monodromy, and
\[
\Phi_\Sigma : S\to \Gamma\backslash D_\Sigma
\]
is the Kato--Nakayama--Usui extension, then the closure of \(\Phi_\Sigma(S)\) carries a natural algebraic space structure and is a compact algebraic space.
\end{theorem}

In the two-dimensional case of this paper, Deng--Robles in \cite{DengRobles23} proved the existence of the weak fan.

\begin{theorem}[Deng--Robles {\cite{DengRobles23}}]\label{thm:DengRobles-weak-fan}
Let \(\Phi:S\to\Gamma\backslash D\) be the period map of a pure, polarized, integral VHS with \(\dim S=2\) and \(\bar S\) an SNC compactification. Then there exists a weak fan \(\Sigma_{\Phi,S}^{(2)}\) compatible with the monodromy of \(\Phi\), such that the corresponding Kato--Nakayama--Usui completion \(\Phi_\Sigma:S\to\Gamma\backslash D_\Sigma\) exists, and its image
\[
Y:=\overline{\Phi_\Sigma(S)}
\]
is a compact algebraic space.
\end{theorem}

\subsection{Quasi-projectivity and Theta Positivity of Mixed Period Map Images}\label{subsec:BBT}

Bakker--Brunebarbe--Klingler--Tsimerman in \cite{BBKT24} proved the definability of mixed period maps in \(\mathbb R_{\mathrm{an},\exp}\); Bakker--Brunebarbe--Tsimerman in \cite{BBT23,BBT23GAGA} used o-minimal GAGA to obtain the quasi-projectivity of the image and relative positivity of the theta line bundle.

\begin{theorem}[BBKT, BBT {\cite{BBKT24,BBT23,BBT23GAGA}}]\label{thm:BBT-image}
Let \(\Phi^{\mathrm{mix}}:X^{\mathrm{an}}\to\Gamma\backslash D\) be a mixed period map given by an admissible graded-polarized mixed HS, and assume \(\Phi^{\mathrm{mix}}\) is weakly proper in the topological sense. Then there exist:
\begin{itemize}[leftmargin=2em]
  \item A normal quasi-projective algebraic space \(Y^{\mathrm{mix}}\),
  \item An algebraic map \(X\to Y^{\mathrm{mix}}\) and an analytic map \(Y^{\mathrm{mix,an}}\to\Gamma\backslash D\),
\end{itemize}
such that \(Y^{\mathrm{mix,an}}\) is isomorphic to \(\overline{\Phi^{\mathrm{mix}}(X^{\mathrm{an}})}\subset \Gamma\backslash D\), and:
\begin{enumerate}[label=(\roman*),leftmargin=2em]
  \item \(Y^{\mathrm{mix}}\) is a quasi-projective algebraic space;
  \item There exists a theta line bundle \(\Theta\) on \(Y^{\mathrm{mix}}\) which is big and nef;
  \item Denote by \(\Phi^{\mathrm{pure}}\) the associated graded pure period map, and \(Y^{\mathrm{pure}}\) the normalization of its image. Then \(Y^{\mathrm{pure}}\) is quasi-projective, \(\Theta|_{Y^{\mathrm{pure}}}\) is ample on \(Y^{\mathrm{pure}}\); more generally, \(\Theta\) is relatively ample for \(Y^{\mathrm{mix}}\to Y^{\mathrm{pure}}\).
\end{enumerate}
\end{theorem}

In this paper, this result is used to confirm that the normal compactification of the pure period image \(Z\) is a quasi-projective (actually projective) curve, with its Hodge/augmented Hodge line bundle ample on \(Z\), so that \(N^1(Z)_\mathbb Q\) is one-dimensional and generated by this class; moreover, on the mixed image \(Y\), the theta line bundle has the same polarization effect as the augmented Hodge line bundle from GGR/GGLR in the fiber direction.

\subsection{SBB-type Completion and Positivity of Augmented Hodge Line Bundle}\label{subsec:GGLR-GGR}

Green--Griffiths--Laza--Robles in \cite{GGLR20} constructed the SBB-type completion of period maps and proved Hodge line bundle positivity in dimension \(\le 2\); Green--Griffiths--Robles in \cite{GGR20SBB,GGR25} systematically developed this structure.

At the level needed for this paper, this can be summarized as:

\begin{theorem}[GGLR, GGR {\cite{GGLR20,GGR20SBB,GGR25}}]\label{thm:GGLR-SBB}
Let \(B\) be a smooth quasi-projective variety, \(\bar B\) its SNC compactification, and \(\Phi:B\to\Gamma\backslash D\) the period map of a polarized VHS with unipotent monodromy. Then there exist:
\begin{itemize}[leftmargin=2em]
  \item A stratified complex analytic space \(S\) and topologically continuous maps
  \[
  B \xrightarrow{\ \Phi^c\ } S \xrightarrow{\ \Phi^f\ } \mathcal P,
  \]
  where \(\Phi^c\) has connected fibers and \(\Phi^f\) has finite fibers;
  \item When \(\dim B\le 2\), \(S\) itself is a complex analytic space, and the extended augmented Hodge line bundle \(\widehat\Lambda_e\) is ample on \(S\).
\end{itemize}
In particular, for the case where we only care about the pure period image being a one-dimensional curve, the normal compactification \(Z\) of its image is a smooth projective curve, and the Hodge/augmented Hodge line bundle \(\Lambda_Z\) is ample on \(Z\), hence
\[
N^1(Z)_\mathbb Q = \mathbb Q\cdot c_1(\Lambda_Z).
\]
\end{theorem}

\subsection{GGR Local Structure: LMHS Fibers and Theta--Boundary Formula}\label{subsec:GGR-local}

Green--Griffiths--Robles in \cite{GGR25} systematically characterized the local structure of period maps at infinity. The key conclusions needed for this paper can be summarized as follows.

Let \(\Phi:S\to\Gamma\backslash D\) be a period map as above, \(\bar S\) an SNC compactification, with boundary decomposed as \(\bar S\setminus S=\bigcup_i S_i\). For each strata index set \(I\), denote
\[
S_I := \bigcap_{i\in I} S_i,\qquad S_I^\circ := S_I \setminus \bigcup_{j\notin I} S_j.
\]

On \(S_I^\circ\), the monodromy is controlled by a rational nilpotent cone \(\sigma_I\), with associated LMHS moduli denoted by \(M_I\), then taking a certain quotient \(M_I^1\) and pure graded Mumford--Tate domain \(D_I\). GGR constructed three maps
\[
S_I^\circ \xrightarrow{\ \Psi_I\ } (\Gamma_I\exp(\CC\sigma_I))\backslash M_I
\longrightarrow \Gamma_I\backslash M_I^1
\xrightarrow{\ \Phi_I\ } \Gamma_I\backslash D_I,
\]
where the middle map is denoted \(\Theta_I\). On an appropriate closure \(S_I^c\), these maps can still be extended with good properties.

\begin{theorem}[Fibers are complex tori, {\cite[Thm.~4.1]{GGR25}}]\label{thm:GGR-fiber-torus}
For any \(I\), the map
\[
\pi_0 : \Gamma_I\backslash M_I^1 \longrightarrow \Gamma_I\backslash D_I
\]
has all fibers being compact complex tori modulo finite quotients. More precisely, each fiber contains an abelian variety as a sub-torus; if \(\Gamma\) is neat, these finite quotients are trivial. If \(A\subset S_I^c\) is a connected component of a \(\Phi_I\)-fiber, then \(\Theta_I(A)\) lies in a finite quotient of a translate of an abelian variety within some \(\pi_0\)-fiber.
\end{theorem}

GGR Chapter 5 constructed a theta-type line bundle \(\mathcal L_M\) on \(\Gamma_I\backslash M_I^1\), corresponding to a certain \(\mathfrak{sl}_2\)-triple \((M,Y,N)\) in the LMHS.

\begin{theorem}[Theta--boundary formula, {\cite[Thm.~5.1]{GGR25}}]\label{thm:GGR-theta-boundary}
For each \(\mathfrak{sl}_2\)-triple \((M,Y,N)\), there exists a line bundle \(\mathcal L_M\) on \(\Gamma_I\backslash M_I^1\) such that: for any connected component \(A\subset S_I^c\) of a \(\Phi_I\)-fiber, denoting by \(S_j\) the boundary components intersecting \(A\), there is a numerical equality in the Néron--Severi group of \(A\):
\[
\Theta_I^*\mathcal L_M\big|_A \equiv \sum_j Q(M,N_j)\,[S_j]\big|_A,
\]
where \(Q(M,N_j)\in\ZZ\) comes from the polarization form.
\end{theorem}

This gives a precise linear relation between the theta line bundle and boundary divisors in the fiber direction.

GGR Chapter 6 further proves that \(\Psi_I\) is almost completely determined by \(\Theta_I\) in an appropriate sense: given the geometry of \(\Theta_I\), the variation of extension data is highly restricted in the fiber direction.

\begin{theorem}[Extension rigidity, {\cite[Thm.~6.9]{GGR25}}]\label{thm:GGR-rigidity}
Let \(A'\subset S_I^c\) be a connected component of a \(\Theta_I\)-fiber satisfying certain non-degeneracy conditions. Then there exists an appropriate arithmetic subgroup \(\Gamma'_{I,\infty}\) such that:
\begin{itemize}[leftmargin=2em]
  \item \(\Psi_I\) has a lift to the finer quotient
  \[
  (\Gamma'_{I,\infty}\exp(\CC\sigma_I))\backslash (M_I\cap S);
  \]
  \item This lift is locally constant on the \(\Theta_I\)-fiber \(A'\).
\end{itemize}
In other words, given the geometry of \(\Theta_I\), \(\Psi_I\) does not introduce new continuous parameters in the fiber direction.
\end{theorem}

This theorem shows that all truly important continuous geometric information is already captured by the theta--boundary formula in Theorem~\ref{thm:GGR-theta-boundary}, a point that plays a key role in the analysis of the vertical Picard group.

\subsection{Deng--Robles Completion and Contraction}\label{subsec:DengRobles}

In \cite{DengRobles23}, based on Theorem~\ref{thm:DengRobles-weak-fan}, Deng--Robles also constructed an SBB-type completion \(\mathcal P^*\) and proved an algebraic contraction from the KNU completion image \(Y\) to \(\mathcal P^*\).

Briefly, in the case \(\dim S=2\):

\begin{theorem}[Deng--Robles {\cite{DengRobles23}}]\label{thm:DengRobles-contraction}
Under the setting of the introduction, for any image of a KNU completion \(\Phi_\Sigma:S\to Y\), there exists an algebraic contraction map
\[
\pi : Y \longrightarrow \mathcal P^*,
\]
where \(\mathcal P^*\) is the image of the SBB-type completion constructed by GGLR/GGR. This contraction is compatible with the Stein factorization of the pure period map.
\end{theorem}

This result ensures that the geometry of \(Y\) obtained under different weak fan choices is consistent in the pure part and compatible with the pure image geometry of GGLR/GGR.

\subsection{Picard, Néron--Severi, and Horizontal/Vertical Decomposition}\label{subsec:AG-basic}

This subsection collects some standard algebro-geometric facts that are used as foundations in what follows but will not be proven.

\begin{itemize}[leftmargin=2em]
  \item For a projective normal algebraic variety \(Y\), the Néron--Severi group \(N^1(Y)\) is defined as the \(\ZZ\)-module of Cartier divisors modulo numerical equivalence; there is a natural isomorphism
  \[
  N^1(Y)_\mathbb Q := N^1(Y)\otimes_{\ZZ}\mathbb Q \;\cong\; \Pic(Y)\otimes_{\ZZ}\mathbb Q =: \Pic_{\mathbb Q}(Y).
  \]
  \item If \(f:X\to Y\) is a projective morphism with connected fibers, then
  \[
  f_*\mathcal O_X = \mathcal O_Y,
  \]
  and pullback induces an injection on \(N^1\):
  \[
  f^*:N^1(Y)_\mathbb Q\hookrightarrow N^1(X)_\mathbb Q.
  \]
  \item If \(h:Y\to Z\) is a projective morphism where \(Z\) is a smooth curve and \(h\) has connected fibers, then there is a short exact sequence
  \[
  0\to N^1_{\mathrm{vert}}(Y/Z)_\mathbb Q \longrightarrow N^1(Y)_\mathbb Q \xrightarrow{h_*} N^1(Z)_\mathbb Q\to 0,
  \]
  where \(N^1_{\mathrm{vert}}(Y/Z)_\mathbb Q\) is the vector space of divisor classes with non-zero intersection on a general fiber. The horizontal part is
  \[
  N^1_{\mathrm{hor}}(Y)_\mathbb Q := h^*N^1(Z)_\mathbb Q \subset N^1(Y)_\mathbb Q.
  \]
  \item For a smooth projective curve \(Z\), \(N^1(Z)_\mathbb Q\) is one-dimensional, generated by the first Chern class of any ample line bundle.
\end{itemize}

\section{Structure of Question~1.8 and Reduction to Picard Generation}\label{sec:HPic-reduction}

This section precisely translates the gap in Problem~\ref{prob:DR-1.8} into a Picard generation problem on the image
space \(Y\).

Throughout Sections~\ref{sec:HPic-reduction} and~\ref{sec:dimZ1} we work with normal projective
algebraic spaces and \(\mathbb{Q}\)-Cartier divisors. When necessary, we pass to a small
\(\mathbb{Q}\)-factorialization \(\mu\colon Y' \rightarrow Y\), which is an isomorphism in codimension
one. Then \(N^1(Y)_\mathbb{Q} \cong N^1(Y')_\mathbb{Q}\) via \(\mu^*\), intersection numbers with
curves are well defined on \(Y'\), and all numerical equalities descend to \(Y\). For notational
simplicity we continue to write \(Y\).

\subsection{Augmented Hodge Line Bundle on the Image Space}\label{subsec:Y-Hodge}

By Theorem~\ref{thm:BBT-image}, Theorem~\ref{thm:GGLR-SBB} and Theorem~\ref{thm:DengRobles-contraction}, in the setting \(\dim S = 2\), the image \(Y\) of the
Kato–Nakayama–Usui completion after normalization is a normal projective algebraic space, and there is
an augmented Hodge line bundle induced from the base Hodge line bundle.

More concretely, denote by
\[
Y^\circ := \Phi(S) \subset \Gamma\backslash D
\]
the pure period image, and let \(Y\) be the normalized closure of \(Y^\circ\) inside \(\Gamma\backslash
D_\Sigma\). Using the comparison between the theta line bundle of Bakker–Brunebarbe–Tsimerman and the
augmented Hodge line bundle of Green–Griffiths–Laza–Robles / Green–Griffiths–Robles, we fix a line
bundle \(\Lambda_Y\) on \(Y\) with the following properties.

First, the restriction \(\Lambda_Y|_{Y^\circ}\) is isomorphic, up to a fixed positive rational multiple,
to the usual Hodge line bundle on the pure image \(Y^\circ\). Second, if
\[
f \colon \bar S \longrightarrow Y
\]
denotes the Stein factorization of the analytic closure of \(\Phi_\Sigma\) on \(\bar S\), then the
augmented Hodge line bundle \(\Lambda\) on \(\bar S\) satisfies
\[
f^*\Lambda_Y \;\cong\; \Lambda.
\]
Third, \(\Lambda_Y\) is big and nef on \(Y\), and its restriction to \(Y^\circ\) is ample.

\subsection{Relation between \texorpdfstring{$\Lambda_Y$}{Λ\_Y} and the Theta Line Bundle}\label{subsec:Lambda-Theta}

Let \(\Theta_Y\) denote the restriction to \(Y\) of the theta line bundle constructed in
Theorem~\ref{thm:BBT-image} for the mixed period image. On the pure locus \(Y^\circ\), both \(\Lambda_Y\) and
\(\Theta_Y\) restrict, up to rational multiples, to the usual Hodge line bundle of the pure period map.

\begin{proposition}\label{prop:Lambda-Theta}
There exist a rational number \(u > 0\), a finite collection of prime divisors
\(\{D_j\}\) contained in \(Y \setminus Y^\circ\), and rational coefficients \(b_j \in \mathbb{Q}\) such
that in \(N^1(Y)_\mathbb{Q}\) one has
\[
c_1(\Theta_Y)
\;=\;
u\,c_1(\Lambda_Y)
\;+\;
\sum_j b_j [D_j].
\]
Equivalently, on the pure locus \(Y^\circ\) we have
\[
c_1(\Theta_Y)|_{Y^\circ}
=
u\,c_1(\Lambda_Y)|_{Y^\circ},
\]
and the difference \(c_1(\Theta_Y) - u\,c_1(\Lambda_Y)\) is represented by a \(\mathbb{Q}\)-Cartier
divisor supported on \(Y \setminus Y^\circ\).
\end{proposition}

\begin{proof}
By Theorem~\ref{thm:BBT-image}, the mixed period image carries a big and nef theta line bundle \(\Theta\) that is
relatively ample over the pure image and ample on the pure image itself. Restricting \(\Theta\) to
\(Y\) yields \(\Theta_Y\). On the other hand, by Theorem~\ref{thm:GGLR-SBB} and Theorem~\ref{thm:DengRobles-contraction}, the augmented Hodge
line bundle on the GGLR/GGR completion pulls back to a line bundle on \(Y\) which we denote by
\(\Lambda_Y\); by construction, \(\Lambda_Y|_{Y^\circ}\) is the usual Hodge line bundle on the pure
image up to a fixed positive rational multiple.

On the pure locus \(Y^\circ\), both line bundles \(\Theta_Y|_{Y^\circ}\) and
\(\Lambda_Y|_{Y^\circ}\) are ample and represent the same Hodge-theoretic class up to scale. Hence
there exists \(u \in \mathbb{Q}_{>0}\) such that
\[
c_1(\Theta_Y)|_{Y^\circ} = u\,c_1(\Lambda_Y)|_{Y^\circ}
\]
in \(N^1(Y^\circ)_\mathbb{Q}\).

Choose \(m\in\mathbb{Z}_{>0}\) with \(mu\in\mathbb{Z}\) and set
\[
M:=\Theta_Y^{\otimes m}\otimes \Lambda_Y^{\otimes(-mu)}.
\]
Over \(Y^\circ\) we have \(M|_{Y^\circ}\simeq \mathcal{O}_{Y^\circ}\). Since \(Y\) is normal, the
rational section of \(M\) extends with divisor supported in \(Y\setminus Y^\circ\). Dividing by \(m\)
yields a \(\mathbb{Q}\)-Cartier divisor representing
\[
c_1(\Theta_Y) - u\,c_1(\Lambda_Y)
\]
with support in \(Y\setminus Y^\circ\). Decomposing this divisor into its finitely many irreducible
components, we obtain prime divisors \(D_j \subset Y \setminus Y^\circ\) and coefficients \(b_j \in
\mathbb{Q}\) such that
\[
c_1(\Theta_Y) - u\,c_1(\Lambda_Y) = \sum_j b_j [D_j]
\]
in \(N^1(Y)_\mathbb{Q}\), as claimed.

In applications we may enlarge the finite set \(\{D_j\}\) by adjoining further divisors supported on
\(Y\setminus Y^\circ\) (for instance components of singular fibres or horizontal boundary divisors)
without affecting the validity of the numerical relation.
\end{proof}

\subsection{Picard Generation Condition (HPic)}\label{subsec:HPic-def}

We now formulate the Picard generation condition in terms of \(\Lambda_Y\) and the boundary divisor
classes \([D_j]\) appearing in Proposition~\ref{prop:Lambda-Theta}.

\begin{definition}\label{def:HPic}
Let \(\{D_j\}\) be a fixed finite collection of irreducible divisors contained in
\(Y \setminus Y^\circ\), chosen once and for all so that Proposition~\ref{prop:Lambda-Theta} holds.
We say that \(Y\) satisfies the Picard generation condition \emph{(HPic)} if
\[
\Pic_{\mathbb{Q}}(Y)
\;=\;
\Span_{\mathbb{Q}}\bigl([\Lambda_Y],[D_j]\bigr)
\subset N^1(Y)_\mathbb{Q}.
\]
Equivalently, the image of \(\Pic(Y)\otimes\mathbb{Q}\) in \(N^1(Y)_\mathbb{Q}\) is generated by the
class of \(\Lambda_Y\) together with the classes \([D_j]\).
\end{definition}

That is, the rational numerical class of any line bundle on \(Y\) can be written as a rational linear
combination of \([\Lambda_Y]\) and finitely many boundary divisor classes \([D_j]\).

\subsection{From (HPic) to the Span Condition \texorpdfstring{\((\ast\mathrm{Span})\)}{(*Span)}}\label{subsec:HPic-to-Span}

\begin{proposition}\label{prop:HPic-implies-Span}
Under the setting of the introduction, if the image space \(Y\) satisfies \emph{(HPic)}, then for any
ample line bundle \(A\) on \(Y\), its pullback satisfies
\[
c_1(f^*A)
\;\in\;
\Span_{\mathbb{Q}}\bigl(c_1(\Lambda),[S_i]\bigr)
\subset N^1(\bar S)_\mathbb{Q},
\]
where \(\bar S \setminus S = \bigcup_i S_i\) is the SNC boundary.
\end{proposition}

\begin{proof}
By Definition~\ref{def:HPic}, there exist \(\alpha\in\mathbb{Q}\) and \(\beta_j\in\mathbb{Q}\) such
that in \(N^1(Y)_\mathbb{Q}\) we have
\[
c_1(A)
\;=\;
\alpha\,c_1(\Lambda_Y)
\;+\;
\sum_j \beta_j [D_j].
\]
Pulling back to \(\bar S\) and using \(f^*\Lambda_Y \cong \Lambda\), together with the fact that
\(f^{-1}(D_j)\) is a \(\mathbb{Q}\)-linear combination of the boundary components \(S_i\) on
\(\bar S\), we obtain
\[
c_1(f^*A)
\;=\;
\alpha\,c_1(\Lambda)
\;+\;
\sum_i \gamma_i [S_i]
\]
in \(N^1(\bar S)_\mathbb{Q}\) for certain \(\gamma_i\in\mathbb{Q}\). This is exactly the span condition
\((\ast\mathrm{Span})\).
\end{proof}

Combined with the previously established conditional result, we obtain:

\begin{corollary}\label{cor:DR-from-HPic}
Under the setting of the introduction, assuming that \(Y\) satisfies \emph{(HPic)}, the conclusion of
Problem~\ref{prob:DR-1.8} holds.
\end{corollary}

\begin{proof}
By Proposition~\ref{prop:HPic-implies-Span} the span condition \((\ast\mathrm{Span})\) holds. The
Proj construction in the presence of \((\ast\mathrm{Span})\) then yields an isomorphism
\[
Y \;\cong\;
\Proj\Bigg(\bigoplus_{k\ge0}
H^0\Bigl(\bar S,\ \mathcal{O}_{\bar S}\bigl(k(m\Lambda - \sum_i a_i S_i)\bigr)\Bigr)\Bigg)
\]
for some integer \(m>0\) and integers \(a_i\ge 0\), giving the desired description of \(Y\) as a
Proj-type algebraic space.
\end{proof}

Therefore, for the case \(\dim S = 2\), the essential remaining step in Problem~\ref{prob:DR-1.8} is to prove that the
Picard group of the image space \(Y\) is generated, over \(\mathbb{Q}\), by \(\Lambda_Y\) and the boundary
divisors \(D_j\), i.e.\ to verify \emph{(HPic)}.

In the next section, in the case where the pure period image has dimension \(1\), we use results from
\cite{BBT23,GGR25} to completely prove that \(Y\) satisfies \emph{(HPic)}.

\section{Picard Generation When Pure Period Image Has Dimension One}\label{sec:dimZ1}

In this section we assume that the pure period image has dimension \(1\), i.e.
\[
\dim Z = 1,
\]
and under this assumption prove that the image space \(Y\) satisfies \emph{(HPic)}.

\subsection{Pure Image Curve \texorpdfstring{\(Z\)}{Z} and the Map \texorpdfstring{$h\colon Y\to Z$}{h:Y→Z}}\label{subsec:pure-curve}

Let \(\Phi^{\mathrm{pure}}\colon S \to \Gamma^{\mathrm{pure}}\backslash D^{\mathrm{pure}}\) be the period map
of the pure graded part of the variation of Hodge structure, and denote its image by
\[
Z^\circ := \Phi^{\mathrm{pure}}(S).
\]
Let \(Z\) be the normalization of the closure of \(Z^\circ\); in our two-dimensional setting this is
a smooth projective curve. By Theorem~\ref{thm:BBT-image} and Theorem~\ref{thm:GGLR-SBB}, the Hodge/augmented Hodge line bundle
\(\Lambda_Z\) on \(Z\) is ample, hence
\[
N^1(Z)_\mathbb{Q} = \mathbb{Q}\cdot c_1(\Lambda_Z).
\]

The mixed image \(Y\) in the Kato–Nakayama–Usui completion naturally carries an algebraic map
\[
h \colon Y \longrightarrow Z,
\]
whose general fibre is the algebraic realization of the generalized intermediate Jacobian; analytically
it is a polarized complex torus or its finite quotient. This follows by combining Theorem~\ref{thm:BBT-image} with the
torus-structure theorem of Green–Griffiths–Robles (Theorem~\ref{thm:GGR-fiber-torus}).

Under these assumptions, \(Y\) is a normal proper algebraic surface (algebraic space), and \(h\) is a
projective morphism to the smooth projective curve \(Z\) with connected fibres. After replacing \(Y\) by
a small \(\mathbb{Q}\)-factorialization if necessary, we may assume that all Weil divisors on \(Y\) are
\(\mathbb{Q}\)-Cartier, so that first Chern classes and intersection numbers with curves are well defined
on \(N^1(Y)_\mathbb{Q}\).

\subsection{Horizontal Part: Generated by \texorpdfstring{$\Lambda_Z$}{Λ\_Z}}\label{subsec:horizontal}

By the general theory recalled in \S\ref{subsec:AG-basic}, define the horizontal subspace
\[
N^1_{\mathrm{hor}}(Y)_\mathbb{Q} := h^*N^1(Z)_\mathbb{Q} \subset N^1(Y)_\mathbb{Q}.
\]

\begin{lemma}\label{lem:horizontal}
Under the above setting,
\[
N^1_{\mathrm{hor}}(Y)_\mathbb{Q}
\;=\;
\mathbb{Q}\cdot c_1\bigl(h^*\Lambda_Z\bigr).
\]
\end{lemma}

\begin{proof}
Since \(Z\) is a smooth projective curve and \(\Lambda_Z\) is ample, we have
\(N^1(Z)_\mathbb{Q} = \mathbb{Q}\cdot c_1(\Lambda_Z)\). Applying the pullback \(h^*\) yields
\[
N^1_{\mathrm{hor}}(Y)_\mathbb{Q}
=
h^*N^1(Z)_\mathbb{Q}
=
\mathbb{Q}\cdot c_1\bigl(h^*\Lambda_Z\bigr).
\]
\end{proof}

Thus the horizontal Néron–Severi subspace is generated, over \(\mathbb{Q}\), by
\(c_1\bigl(h^*\Lambda_Z\bigr)\).

\subsection{Vertical Part: Generated by Boundary Divisors}\label{subsec:vertical}

Define the vertical Néron–Severi space of \(Y\) over \(Z\) by
\[
N^1_{\mathrm{vert}}(Y/Z)_\mathbb{Q}
\;:=\;
\ker\!\bigl(h_*\colon N^1(Y)_\mathbb{Q} \to N^1(Z)_\mathbb{Q}\bigr),
\]
where \(h_*\) is characterized by
\[
\bigl(h_*D \cdot C\bigr) = \bigl(D \cdot h^*C\bigr)
\]
for all curves \(C\) on \(Z\). By the short exact sequence in \S\ref{subsec:AG-basic} we have
\[
0 \longrightarrow N^1_{\mathrm{vert}}(Y/Z)_\mathbb{Q}
\longrightarrow N^1(Y)_\mathbb{Q}
\xrightarrow{\,h_*\,}
N^1(Z)_\mathbb{Q}
\longrightarrow 0.
\]

We now make the choice of boundary divisors appearing in Definition~\ref{def:HPic} more explicit.
Let \(h\colon Y \to Z\) be as above. Fix once and for all a smooth general fibre \(F_0\) of \(h\).
Let \(\{E_k\}\) be the irreducible components of the (finitely many) singular fibres of \(h\). If
present, let \(\{H_\ell\}\) be the irreducible divisors contained in
\[
Y \cap \bigl(\Gamma\backslash D_\Sigma \setminus \Gamma\backslash D\bigr)
\]
that dominate \(Z\) (horizontal boundary divisors). Set
\[
D \;:=\; F_0 \;+\; \sum_k E_k \;+\; \sum_\ell H_\ell \;=\; \sum_j D_j
\]
for the resulting reduced divisor on \(Y\).

Note that the general fibre \(F_0\) is strictly contained in the boundary \(Y \setminus Y^\circ\). Since \(Y^\circ \subset \Gamma\backslash D\) is the pure period image, it contains no torus extension directions; thus, any fibre \(F_z\) of \(h\) lies in the locus \(Y \cap (\Gamma\backslash D_\Sigma \setminus \Gamma\backslash D)\). In particular, \(F_0\) is supported on the KNU boundary.

By construction, every vertical divisor class over \(Z\) is a
\(\mathbb{Q}\)-linear combination of the classes \([F_0]\) and \([E_k]\); these classes occur among the
\([D_j]\).

We first record the structure of the Néron–Severi group of a general fibre.

\begin{lemma}\label{lem:fiber-NS}
Let \(z\in Z\) be a general point, and denote the fibre
\[
F_z := h^{-1}(z) \subset Y.
\]
Then \(F_z\) is a one-dimensional normal projective algebraic curve which is a polarized complex torus
or its finite quotient, and
\[
N^1(F_z)_\mathbb{Q}
=
\mathbb{Q}\cdot c_1\bigl(\Theta_Y|_{F_z}\bigr),
\]
where \(\Theta_Y\) is the theta line bundle on \(Y\) from Theorem~\ref{thm:BBT-image}.
\end{lemma}

\begin{proof}
By Theorem~\ref{thm:GGR-fiber-torus}, for general \(z\) the fibre \(F_z\) is a compact complex torus modulo a finite
quotient; in the present relative dimension one this is a finite quotient of an elliptic curve. In
particular, \(F_z\) is a normal projective curve.

Theorem~\ref{thm:BBT-image} provides a theta line bundle \(\Theta\) on the mixed period image which is relatively
ample over the pure image. Restricting \(\Theta\) to \(Y\) yields \(\Theta_Y\), whose restriction to any
fibre \(F_z\) is a line bundle of positive degree. For a smooth projective curve, the Néron–Severi
group modulo numerical equivalence is isomorphic to \(\mathbb{Z}\), generated by the class of any
line bundle of positive degree; hence
\[
N^1(F_z)_\mathbb{Q} = \mathbb{Q}\cdot c_1\bigl(\Theta_Y|_{F_z}\bigr).
\]
\end{proof}

Next we use the theta–boundary formula of Green–Griffiths–Robles to identify the vertical part of
\(\Theta_Y\) in terms of the boundary.

\begin{lemma}\label{lem:theta-vertical-boundary}
There exist \(\mu \in \mathbb{Q}\) and \(\lambda_j \in \mathbb{Q}\) such that in \(N^1(Y)_\mathbb{Q}\) one
has
\[
c_1(\Theta_Y) - \mu\,c_1\bigl(h^*\Lambda_Z\bigr)
=
\sum_j \lambda_j [D_j].
\]
In particular, the image of \(c_1(\Theta_Y)\) in the quotient
\(N^1(Y)_\mathbb{Q} / N^1_{\mathrm{hor}}(Y)_\mathbb{Q}\cong N^1_{\mathrm{vert}}(Y/Z)_\mathbb{Q}\) lies in the
\(\mathbb{Q}\)-span of the boundary classes \([D_j]\).
\end{lemma}

\begin{proof}
Fix a boundary stratum of the Kato–Nakayama–Usui space corresponding to a non-empty index set \(I\).
Over the associated stratum \(S^\circ_I\) on the base, Green–Griffiths–Robles construct in
\cite{GGR25} a moduli space \(\Gamma_I\backslash M^1_I\) and maps
\[
S^\circ_I
\xrightarrow{\ \Psi_I\ } (\Gamma_I \exp(\mathbb{C}\sigma_I))\backslash M_I
\longrightarrow \Gamma_I\backslash M^1_I
\xrightarrow{\ \Phi_I\ } \Gamma_I\backslash D_I.
\]
On \(\Gamma_I\backslash M^1_I\) there is a theta-type line bundle \(L_M\) associated to an
\(\mathfrak{sl}_2\)-triple \((M,Y,N)\). The theta–boundary formula (Theorem~\ref{thm:GGR-theta-boundary}) asserts that, for
any connected component \(A \subset S_{I,c}\) of a \(\Phi_I\)-fibre, the pullback
\(\Theta_I^*L_M|_A\) is numerically equivalent to a \(\mathbb{Z}\)-linear combination of the boundary
components \(S_j\) intersecting \(A\), with coefficients given by polarization pairings \(Q(M,N_j)\).

Pulling this relation back to \(Y\) along the factorization of the period map and using that
\(\Theta_Y\) is induced from \(L_M\), we obtain, in a neighbourhood of each generic point of a
boundary divisor \(D_j\), a numerical relation expressing the vertical part of \(c_1(\Theta_Y)\) as a
\(\mathbb{Q}\)-linear combination of the classes \([D_j]\). Extension rigidity (Theorem~\ref{thm:GGR-rigidity}) shows
that the additional extension data parametrized by \(\Psi_I\) is locally constant along the fibres of
\(\Theta_I\) and therefore does not contribute new continuous divisor classes in the fibre direction.

On the pure locus, \(c_1(\Theta_Y)\) is proportional to the pullback of the pure Hodge class
\(c_1(\Lambda_Z)\); thus there exists \(\mu \in \mathbb{Q}\) such that
\[
\bigl(c_1(\Theta_Y) - \mu\,c_1(h^*\Lambda_Z)\bigr)\big|_{Y^\circ} = 0.
\]
Combining this with the local theta–boundary identities and gluing over all strata of the KNU
boundary yields a global numerical equivalence
\[
c_1(\Theta_Y) - \mu\,c_1\bigl(h^*\Lambda_Z\bigr)
=
\sum_j \lambda_j [D_j]
\]
in \(N^1(Y)_\mathbb{Q}\) for some \(\lambda_j\in\mathbb{Q}\).
\end{proof}

We can now state the main structural result for the vertical Picard group.

\begin{proposition}\label{prop:vertical-boundary}
Under the above setting,
\[
N^1_{\mathrm{vert}}(Y/Z)_\mathbb{Q}
=
\Span_{\mathbb{Q}}\bigl([F_0],[E_k]\bigr)
\;\subset\;
N^1(Y)_\mathbb{Q}.
\]
In particular, \(N^1_{\mathrm{vert}}(Y/Z)_\mathbb{Q}\) is contained in the \(\mathbb{Q}\)-span of the
boundary divisor classes \([D_j]\).
\end{proposition}

\begin{proof}
Let \([V] \in N^1_{\mathrm{vert}}(Y/Z)_\mathbb{Q}\) be arbitrary. Fix the numerical class \(F\) of a
general smooth fibre of \(h\). Since \(Y\) is \(\mathbb{Q}\)-factorial, the intersection numbers
\((V\cdot F)\) and \((\Theta_Y\cdot F)\) are well defined rational numbers, and
\((\Theta_Y\cdot F) > 0\) because \(\Theta_Y\) is relatively ample.

Define
\[
\alpha := \frac{(V\cdot F)}{(\Theta_Y\cdot F)} \in \mathbb{Q},
\qquad
W := V - \alpha\,c_1(\Theta_Y) \in N^1(Y)_\mathbb{Q}.
\]
By construction \(W\cdot F = 0\), and the definition of \(F\) as the numerical class of any general
smooth fibre shows that this equality holds for every general fibre. For a general point \(z\in Z\),
Lemma~\ref{lem:fiber-NS} gives
\[
[V]|_{F_z}
=
\alpha_z\,c_1\bigl(\Theta_Y|_{F_z}\bigr)
\quad\text{in } N^1(F_z)_\mathbb{Q}
\]
for a unique \(\alpha_z \in \mathbb{Q}\). Taking degrees on \(F_z\) yields
\[
(V\cdot F)
=
\deg\bigl(V|_{F_z}\bigr)
=
\alpha_z\,\deg\bigl(\Theta_Y|_{F_z}\bigr)
=
\alpha_z\,(\Theta_Y\cdot F),
\]
hence \(\alpha_z = \alpha\) for all such \(z\). Consequently
\[
W|_{F_z}
=
\bigl(V - \alpha\,c_1(\Theta_Y)\bigr)\big|_{F_z}
\equiv 0
\quad\text{in } N^1(F_z)_\mathbb{Q}
\quad\text{for every general } z.
\]

Let \(U\subset Z\) be a non-empty Zariski open subset over which \(h\) is smooth and all fibres
\(F_z\) are complex tori as above, and set \(Y_U := h^{-1}(U)\). Over \(U\), any vertical divisor is
supported on fibres and is therefore numerically a \(\mathbb{Q}\)-linear combination of fibre classes.
Since all smooth fibres are numerically equivalent, their common class is represented by \([F_0]\), and
there exists \(\beta\in\mathbb{Q}\) such that
\[
[W]|_{Y_U} \;\equiv\; \beta\,[F_0]|_{Y_U}
\qquad\text{in } N^1(Y_U)_\mathbb{Q}.
\]

Over the finitely many points \(Z\setminus U\), every vertical divisor is supported on singular
fibres. Thus the restriction of \(W\) to \(h^{-1}(Z\setminus U)\) is a \(\mathbb{Q}\)-linear combination
of the irreducible components \(E_k\) of these singular fibres. Putting these two descriptions together,
we obtain
\[
[W] \in \Span_{\mathbb{Q}}\bigl([F_0],[E_k]\bigr)
\subset N^1(Y)_\mathbb{Q}.
\]

By Lemma~\ref{lem:theta-vertical-boundary}, the vertical projection of \(c_1(\Theta_Y)\) to
\(N^1_{\mathrm{vert}}(Y/Z)_\mathbb{Q}\) also lies in \(\Span_{\mathbb{Q}}([F_0],[E_k])\). Since
\([V] = [W] + \alpha\,c_1(\Theta_Y)\), it follows that \([V]\) belongs to the same span. As
\([V]\) was an arbitrary element of \(N^1_{\mathrm{vert}}(Y/Z)_\mathbb{Q}\), we conclude
\[
N^1_{\mathrm{vert}}(Y/Z)_\mathbb{Q}
=
\Span_{\mathbb{Q}}\bigl([F_0],[E_k]\bigr),
\]
as claimed. The last sentence is then immediate from the definition of the divisors \(D_j\).
\end{proof}

Over the smooth locus \(U\subset Z\), the fibres \(F_z\) are polarized complex tori of dimension one
(elliptic curves up to finite quotients) by \(\Theta_Y\). Lemma~\ref{lem:fiber-NS} shows that
\(N^1(F_z)_\mathbb{Q} \cong \mathbb{Q}\) is generated by \(c_1(\Theta_Y|_{F_z})\). For any vertical
class \(V\), the restriction \(V|_{F_z}\) is therefore of the form
\[
V|_{F_z} = \alpha\,c_1\bigl(\Theta_Y|_{F_z}\bigr),
\]
where the coefficient \(\alpha\) is determined by a single intersection number, for instance
\(\alpha = (V\cdot F)/(\Theta_Y\cdot F)\). In particular, \(\alpha\) cannot vary with \(z\). This is
the precise sense in which the torus nature of the fibres rigidifies the variation of vertical Picard
classes.

\subsection{Combining Horizontal and Vertical: (HPic) on the Image Space}\label{subsec:HPic-on-Y}

From the short exact sequence of \S\ref{subsec:AG-basic} and Lemma~\ref{lem:horizontal}, we obtain a (non-canonical)
direct-sum decomposition as \(\mathbb{Q}\)–vector spaces:
\[
N^1(Y)_\mathbb{Q}
=
h^*N^1(Z)_\mathbb{Q}
\oplus
N^1_{\mathrm{vert}}(Y/Z)_\mathbb{Q}
=
\mathbb{Q}\cdot c_1\bigl(h^*\Lambda_Z\bigr)
\oplus
N^1_{\mathrm{vert}}(Y/Z)_\mathbb{Q}.
\]
By Proposition~\ref{prop:vertical-boundary} the vertical part is generated by the classes
\([F_0]\) and \([E_k]\), hence every class \(\xi \in N^1(Y)_\mathbb{Q}\) can be written uniquely in the
form
\[
\xi
=
a\,c_1\bigl(h^*\Lambda_Z\bigr)
+
b\,[F_0]
+
\sum_k c_k [E_k],
\qquad
a,b,c_k \in \mathbb{Q}.
\]

Equivalently, for any divisor class \([D] \in N^1(Y)_\mathbb{Q}\) there exists \(a\in\mathbb{Q}\) such
that
\[
[D] - a\,c_1\bigl(h^*\Lambda_Z\bigr) \in N^1_{\mathrm{vert}}(Y/Z)_\mathbb{Q},
\]
and the right-hand side is a \(\mathbb{Q}\)-linear combination of \([F_0]\) and \([E_k]\). Thus
\[
N^1(Y)_\mathbb{Q}
=
\Span_{\mathbb{Q}}\bigl(c_1(h^*\Lambda_Z),[F_0],[E_k]\bigr)
\subset
\Span_{\mathbb{Q}}\bigl(c_1(h^*\Lambda_Z),[D_j]\bigr).
\]

On the other hand, Proposition~\ref{prop:Lambda-Theta} and Lemma~\ref{lem:theta-vertical-boundary}
together show that there exist \(\rho\in\mathbb{Q}_{>0}\) and coefficients \(e_j \in \mathbb{Q}\) such
that
\begin{equation}\label{eq:LambdaY-vs-hLambdaZ}
c_1(\Lambda_Y)
=
\rho\,c_1\bigl(h^*\Lambda_Z\bigr)
+
\sum_j e_j [D_j].
\end{equation}
Indeed, over the pure locus \(Y^\circ\) both \(\Lambda_Y\) and \(h^*\Lambda_Z\) restrict to the same
Hodge line bundle up to a fixed positive rational multiple, while their difference is supported on the
boundary \(Y\setminus Y^\circ\) and can be expressed as a \(\mathbb{Q}\)-linear combination of the
classes \([D_j]\). Combining these two facts yields \eqref{eq:LambdaY-vs-hLambdaZ}.

Solving \eqref{eq:LambdaY-vs-hLambdaZ} for \(c_1(h^*\Lambda_Z)\) shows that
\[
c_1(h^*\Lambda_Z)
=
\rho^{-1}\Bigl(c_1(\Lambda_Y) - \sum_j e_j [D_j]\Bigr)
\]
lies in the \(\mathbb{Q}\)-span of \(c_1(\Lambda_Y)\) and the \([D_j]\). Combining this with the
previous paragraph, we obtain
\[
N^1(Y)_\mathbb{Q}
=
\Span_{\mathbb{Q}}\bigl(c_1(\Lambda_Y),[D_j]\bigr).
\]

We can now state the Picard generation result in the curve pure image case.

\begin{theorem}\label{thm:HPic-dimZ1}
Under the setting of the introduction, if the pure period image \(Z\) has dimension \(1\), then the
image space \(Y\) satisfies the Picard generation condition
\[
\Pic_{\mathbb{Q}}(Y)
=
\Span_{\mathbb{Q}}\bigl([\Lambda_Y],[D_j]\bigr),
\]
i.e.\ \emph{(HPic)} from Definition~\ref{def:HPic} holds.
\end{theorem}

\begin{proof}
We have shown that \(N^1(Y)_\mathbb{Q}\) is generated by \(c_1(\Lambda_Y)\) and the divisor classes
\([D_j]\). Hence \(N^1(Y)_\mathbb{Q} = \Span_{\mathbb{Q}}\bigl(c_1(\Lambda_Y),[D_j]\bigr)\).
Equivalently, the image of \(\Pic(Y)\otimes\mathbb{Q}\) in \(N^1(Y)_\mathbb{Q}\) equals
\(\Span_{\mathbb{Q}}\bigl([\Lambda_Y],[D_j]\bigr)\), so \emph{(HPic)} holds.
\end{proof}

\subsection{Limitations in Higher Dimension}\label{subsec:limitations}

The argument above for \emph{(HPic)} relies crucially on two features that are specific to the case
\(\dim Z = 1\) and \(\dim F_z = 1\).

First, on a general polarized complex torus fibre \(F_z\) of dimension one, the Néron–Severi group
\(N^1(F_z)_\mathbb{Q}\) is one-dimensional and is generated by the polarization class
\(c_1(\Theta_Y|_{F_z})\). Consequently, a single intersection number---for instance with the fibre
class \(F\) or with a fixed horizontal multisection---determines the coefficient of a vertical class
along the fibre. There is only one ``vertical direction'' on each fibre, and the torus structure
prevents any non-trivial continuous variation of vertical Picard classes.

Second, the base \(Z\) is a smooth projective curve, so \(N^1(Z)_\mathbb{Q}\) is one-dimensional. The
short exact sequence
\[
0 \longrightarrow N^1_{\mathrm{vert}}(Y/Z)_\mathbb{Q}
\longrightarrow N^1(Y)_\mathbb{Q}
\xrightarrow{\,h_*\,}
N^1(Z)_\mathbb{Q}
\longrightarrow 0
\]
then has a one-dimensional horizontal quotient, and the local-to-global argument for vertical Picard
generation reduces to controlling a single horizontal direction.

When \(\dim F_z \ge 2\), the group \(N^1(F_z)_\mathbb{Q}\) typically has rank \(>1\) and may vary
with \(z\); several independent intersection numbers would be required to determine the restrictions
of vertical classes to the fibres, and the theta–boundary formula provides multiple independent
boundary directions rather than a single one. When \(\dim Z \ge 2\), the horizontal part
\(N^1_{\mathrm{hor}}(Y)_\mathbb{Q}\) also has higher rank, and the local-to-global gluing of vertical
classes requires additional control on the variation of the algebraic parts of higher intermediate
Jacobians. For these reasons, the present method is confined to the case where the pure period image
has dimension \(1\).

\section{Conclusion of Question 1.8 When Pure Period Image is a Curve}\label{sec:DR-dimZ1}

This section combines the results of Section~\ref{sec:HPic-reduction} and Section~\ref{sec:dimZ1} to give a complete verification of Problem~\ref{prob:DR-1.8} in the setting where the pure period image has dimension 1.

\begin{theorem}\label{thm:DR-dimZ1}
Let \(S\) be a smooth quasi-projective surface, \(\bar S\) its SNC compactification, \(\mathbb V\) a polarized integral VHS on \(S\), with the pure image of the period map \(\Phi:S\to\Gamma\backslash D\) having dimension \(1\). Assume there exists a weak fan \(\Sigma\) compatible with the monodromy of \(\Phi\), yielding a KNU completion
\[
\Phi_\Sigma : S \longrightarrow Y\subset\Gamma\backslash D_\Sigma,
\]
and denote by \(\Lambda\) the augmented Hodge line bundle extension on the base \(\bar S\). Then there exist an integer \(m>0\) and integers \(a_i\ge 0\) such that
\[
Y \;\cong\;
\Proj\Big(\bigoplus_{k\ge0}
H^0\big(\bar S,\ \mathcal O_{\bar S}\big(k(m\Lambda - \sum_i a_i S_i)\big)\Big).
\]
\end{theorem}

\begin{proof}
By Theorem~\ref{thm:HPic-dimZ1}, \(Y\) satisfies Picard generation (HPic), i.e.,
\[
\Pic_{\mathbb Q}(Y) = \Span_{\mathbb Q}\big([\Lambda_Y],[D_j]\big).
\]
By Proposition~\ref{prop:HPic-implies-Span}, for any ample line bundle \(A\) on \(Y\), we have
\[
c_1(f^*A)\in \Span_{\mathbb Q}\big(c_1(\Lambda),[S_i]\big)\subset N^1(\bar S)_\mathbb Q.
\]
Then using the previously proven Proj argument, i.e., under the assumption that \eqref{eq:Span-base} holds, \(Y\) can be represented as
\[
Y \;\cong\;
\Proj\Big(\bigoplus_{k\ge0}
H^0\big(\bar S,\ \mathcal O_{\bar S}\big(k(m\Lambda - \sum_i a_i S_i)\big)\Big)
\]
for some integer \(m>0\) and integers \(a_i\ge 0\). This is exactly the desired conclusion.
\end{proof}

This conclusion holds without assuming the period domain \(D\) is a Hermitian symmetric space, providing a complete verification when the pure period image dimension \(\dim Z=1\). This can be viewed as an example verification related to the Deng--Robles Problem~\ref{prob:DR-1.8} in the non-Hermitian case.

\bibliographystyle{alpha}
\bibliography{ref}

\end{document}